\documentclass[12pt,reqno]{amsart}
\usepackage{amsmath, amsthm, amssymb, cite, bm, enumitem,bbm,color}

\advance \topmargin by -\headheight
\advance \topmargin by -\headsep
     
\evensidemargin \oddsidemargin
\newtheorem{theorem}{Theorem}[section]
\newtheorem{lemma}[theorem]{Lemma}
\newtheorem{corollary}[theorem]{Corollary}

\theoremstyle{definition}

\theoremstyle{remark}

\numberwithin{equation}{section}

\newcommand{\abs}[1]{\lvert#1\rvert}
\newcommand{\mmod}[1]{\,\,(\text{mod}\,\,#1)}

\title{Efficient congruencing in ellipsephic sets:\\the quadratic case}
\author{Kirsti D. Biggs}
\address{Mathematical Sciences, University of Gothenburg and Chalmers Institute of Technology, 412 96 G\"oteborg, Sweden}
\email{biggs@chalmers.se}
\keywords{\hspace{-0.01in}Hardy--Littlewood\hspace{-0.01in} method, \hspace{-0.005in}efficient\hspace{-0.01in} congruencing,\hspace{-0.005in} missing\hspace{-0.01in} digits}

\begin{document}

\begin{abstract}
In this paper, we bound the number of solutions to a quadratic Vinogradov system of equations in which the variables are required to satisfy digital restrictions in a given base. 
Certain sets of permitted digits, namely those giving rise to few representations of natural numbers as sums of elements of the digit set, allow us to obtain better bounds than would be possible using the size of the set alone.
\end{abstract}
\maketitle

\section{Introduction}
Vinogradov's mean value theorem seeks to bound the number of solutions, for a fixed integer $k\geq 2$, to the system of Diophantine equations
\begin{equation}\label{VMVT}
x_1^j+\dots+x_s^j=y_1^j+\dots+y_s^j,\quad (1\leq j\leq k),
\end{equation}
where $1\leq x_i,y_i\leq X$ for all $i$. In this paper, we investigate variants of this problem in which the variables are restricted to certain subsets of the natural numbers which give us significantly stronger control over the associated mean value estimates. Specifically, the subsets of interest to us are defined by digital restrictions---further discussion requires some definitions.

Fix a subset $A\subset\mathbb{N}\cup\{0\}$ with the property that, writing $A_p=A\cap[0,p-1]$, and assuming that $2\leq \#A_p\leq p-1$, we have
\begin{equation}\label{keyAP}
\#\{(a_1,\dots,a_t)\in A_p^t\mid a_1+\dots+a_t=n\}\ll p^{\delta}
\end{equation}
for some $t\geq 2$ and some $\delta>0$, uniformly in $p$. Let
\begin{equation}\label{ellset}
\mathcal{E}=\mathcal{E}_p^{A}=\{n\in\mathbb{N}\mid \textstyle{n=\sum_i a_ip^i} \mbox{ with } a_i\in A_p\mbox{ for all }i\}
\end{equation}
be the set of natural numbers whose base $p$ expansion includes only digits from $A$. Let $I_s(X)=I_{s,2}(X)$ be the number of solutions to (\ref{VMVT}) in the case $k=2$ with $x_i,y_i\in\mathcal{E}(X)=\mathcal{E}\cap[1,X]$ for all $i$, and write $Y$ for $\#\mathcal{E}(X)$. 
\begin{theorem}\label{basicthm}
We have
\begin{equation*}
I_{s}(X)\ll X^{3\delta+\epsilon}(Y^s+Y^{2s-3t}).
\end{equation*}
\end{theorem}

An estimate for the count of solutions to the Vinogradov system (\ref{VMVT}) in the case $k=2$ follows from the quadratic identity 
\begin{equation*}
(a+b-c)^2-(a^2+b^2-c^2)=2(a-c)(b-c)
\end{equation*}
and a standard bound for the divisor function, while recent progress has led to an optimal upper bound in the case of general $k$. When $k=3$, this bound was proved by Wooley in \cite{wooleyk3}, and when $k\geq 4$, by Bourgain, Demeter and Guth in \cite{BDG}, using the $l^2$-decoupling method, and subsequently by Wooley in \cite{NEC}, using the nested efficient congruencing method. These two methods are held to be, respectively, real and $p$-adic analogues of each other---see \cite{ecvl2} for further discussion of this.

We call our set (\ref{ellset})---or, interchangeably, its elements---ellipsephic. This terminology mimics the word \emph{ellips\'ephique}, used in the French mathematical literature to denote integers with missing digits---for example, by Aloui in \cite{Aloui}, and by Aloui, Mauduit and Mkaouar in \cite{AlouiMMk}. The term was coined by Mauduit (see the discussion on page 12 of \cite{colthesis}), although such integers were already studied prior to its introduction.

Writing $r=\#A_p$ for the number of permitted digits, we observe that the cases $r=0$ and $A_p= \{0\}$ are trivial, and the case $r=p$ reduces to the classical case, while when $r=1$ (and $A_p\neq \{0\}$), we see that $\mathcal{E}$ has different behaviour, with $\#\mathcal{E}(X)\approx\log_p{X}$. Consequently, we implement the restriction mentioned above that $2\leq r\leq p-1$, and note that
\begin{align*}
\#\mathcal{E}(X)\ll r^{\log_p{X}+1}=rX^{\log_p{r}},
\end{align*}
and consequently that $\mathcal{E}$ is a thin subset of the integers, in the sense that
\begin{equation*}
\lim_{X\to\infty}\frac{\#\mathcal{E}(X)}{X}=0.
\end{equation*}
We observe that ellipsephic sets have a self-similar, fractal-like structure, with the digital restrictions seen here reminiscent of those in the classical Cantor set. They bear a resemblance to certain real fractal subsets studied by \L aba and Pramanik in \cite{LabaPram}, 
 and by \L aba and Wang in \cite{Cantorrestrn}.

The bounds we obtain in this paper are heavily dependent on the additive structure of the digit set $A$, in a way which we expand on here. A generalised Sidon set, or $B_h[g]$-set, is a subset of the natural numbers in which there are at most $g$ representations of a given $n\in\mathbb{N}$ as the sum of $h$ elements of the set, where representations are counted up to permutation. The sets we are interested in, as suggested earlier in this section by the condition (\ref{keyAP}), can be considered as a further generalisation of this concept.

For $t\geq 2$ an integer, we call a set $A\subset\mathbb{N}_0$ an $E_t(\delta)$-set if (\ref{keyAP}) holds for $\delta>0$ a real number, and we call $A$ an $E_t^*$-set if (\ref{keyAP}) holds for all $\delta>0$.  We refer to a set $\mathcal{E}=\mathcal{E}^A_p$ as a $(p,t,\delta)$-ellipsephic set if $A$ is an $E_t(\delta)$-set, and as a $(p,t)^*$-ellipsephic set if $A$ is an $E_t^*$-set. We now introduce some further notation to allow us to state the more general case of our main result.
For a sequence $\bm{\mathfrak{a}}=(\mathfrak{a}_x)_{x\in\mathcal{E}}$ of complex weights, 
we let
\begin{equation*}
J_{s}(X)=J_{s,2}(X;\bm{\mathfrak{a}})=\oint \Big|\sum_{x\in\mathcal{E}(X)}\mathfrak{a}_x e(\alpha_1 x+ \alpha_2 x^2)\Big|^{2s}\,d\bm{\alpha},
\end{equation*}
where $e(z)$ is shorthand for $e^{2\pi iz}$, and $\oint$ denotes the integral over the unit square $[0,1]^2$. Then $J_{s}(X)$ counts the solutions, in positive integers $x_i,y_i\in\mathcal{E}(X)$, to the system
\begin{equation}\label{quadVin}
x_1^j+\dots+x_s^j=y_1^j+\dots+y_s^j,\quad (1\leq j\leq 2),
\end{equation}
where each solution is counted with weight $\mathfrak{a}_{\bm{x}}\overline{\mathfrak{a}_{\bm{y}}}=\mathfrak{a}_{x_1}\dots\mathfrak{a}_{x_s}\overline{\mathfrak{a}_{y_1}\dots\mathfrak{a}_{y_s}}$. We adopt the convention throughout that statements involving $\epsilon$ hold for any suitably small choice of $\epsilon>0$, and as such the exact value may change from line to line. The vector notation $\bm{x}\equiv\xi\mmod{q}$ means that $x_i\equiv\xi\mmod{q}$ for all $i$, and $\bm{x}\equiv\bm{y}\mmod{q}$ means that $x_i\equiv y_i\mmod{q}$ for all $i$.

Our main theorem provides the following upper bound for $J_{s}(X)$.
\begin{theorem}\label{JThm} 
For $t\geq 2$ an integer, $\delta>0$ a real number, and $p>2$ a prime, let $\mathcal{E}$ be a $(p,t,\delta)$-ellipsephic set and let $Y=\#\mathcal{E}(X)$. 
Then for $s\geq 3t$ we have
\begin{equation*}
J_{s}(X)\ll Y^{s-3t}X^{3\delta+\epsilon}\bigg(\sum_{x\in\mathcal{E}(X)}\left|\mathfrak{a}_x \right|^2\bigg)^s.
\end{equation*}
When $\mathcal{E}$ is a $(p,t)^*$-ellipsephic set, we therefore have
\begin{equation*}
J_{s}(X)\ll Y^{s-3t}X^{\epsilon}\bigg(\sum_{x\in\mathcal{E}(X)}\left|\mathfrak{a}_x \right|^2\bigg)^s.
\end{equation*}
\end{theorem}
Note that it follows from a standard application of H\"older's inequality that for $s\leq 3t$, we have
\begin{equation*}
J_{s}(X)\ll X^{\delta s/t+\epsilon}\bigg(\sum_{x\in\mathcal{E}(X)}\left|\mathfrak{a}_x \right|^2\bigg)^s.
\end{equation*}

\begin{corollary}\label{unwgtd}
Theorem \ref{basicthm} is true.
\end{corollary}
\begin{proof}
This is the case where $\mathfrak{a}_x=1$ for all $x\in\mathcal{E}(X)$.
\end{proof}

The best upper bound which could previously be obtained for $J_{s}(X)$ is a consequence of a result of Bourgain in \cite{bourgainSchr}. Taking $\mathfrak{a}_x=0$ for $x\notin\mathcal{E}$ in that theorem yields, for $s\geq 3$,
\begin{equation*}
J_{s}(X)\ll Y^{s-3}X^{\epsilon}\bigg(\sum_{x\in\mathcal{E}(X)}\left|\mathfrak{a}_x \right|^2\bigg)^s,
\end{equation*}
so we see that a power saving in $Y$ has been obtained by accounting for the specific structure of our ellipsephic sets, rather than just their density.

The proof of Theorem \ref{JThm} uses a version of Wooley's efficient congruencing method which we outline briefly here. We begin by postulating that $J_s(X)$ is significantly larger than the bound asserted in Theorem \ref{JThm}, and proceed to derive a contradiction. We partition our variables into congruence classes modulo powers of the base $p$, and apply H\"older's inequality to restrict our variables to lie in certain common congruence classes. The binomial theorem allows us to convert our equations into congruences featuring a subset of our variables, and using their ellipsephic nature and the $E_t(\delta)$ property, we can ``lift'' solutions to these congruences, at a small cost, to diagonal solutions in which each pair of variables is mutually congruent modulo the relevant power of $p$. Iterating this process, we strengthen the congruences satisfied by these variables---this may be viewed as a ``$p$-adic concentration'' argument, since our variables become closer $p$-adically. By iterating sufficiently many times, we find that our initial assumption on $J_s(X)$ is untenable, which leads us to a contradiction.

We would also like to note the recent paper \cite{Zetal} by Chang, de Dios Pont, Greenfeld, Jamneshan, Li and Madrid, in which the authors prove a decoupling analogue of Theorem \ref{JThm} for real fractal subsets of the parabola. This provides another example of the developing connections between these specific areas of analytic number theory and harmonic analysis, as mentioned above.

We expect the full nested efficient congruencing method, as presented in \cite{NEC}, to deliver appropriate bounds for the ellipsephic version of any other system to which these techniques apply. In particular, the present author has recently obtained the equivalent result, in the case of $E_t^*$-sets, for the number of ellipsephic solutions to (\ref{VMVT}) in the general case $k\geq 3$, which appears in \cite{genellips}. Nevertheless, we believe that the level of detail included in this paper, as well as the treatment of the more general $E_t(\delta)$-sets, merits a full and separate presentation.

Theorem \ref{JThm} has potential applications to a number of Diophantine problems, most notably Waring's problem, in which we attempt to write all natural numbers as sums of a bounded number of squares of ellipsephic integers. A more tractable form of this problem is to seek solutions to
\begin{align*}
n=x_1^2+\dots+x_s^2+y^2,
\end{align*}
with $x_1,\dots, x_s\in\mathcal{E}$ and $y\in\mathbb{N}_0$, which will form a subject for our future work.

As a corollary of Theorem \ref{JThm}, we provide a lower bound on the number of integers representable in the form required by Waring's problem. We would expect to need the set $\mathcal{E}(X)$ to be sufficiently large to give any chance of being able to represent a significant proportion of the integers up to $X$, and as such we incorporate this as an extra condition in the below result.

Let $N_{s}(X)=N_{s,2}^{\mathcal{E}}(X)$ be the number of integers $n$ with $1\leq n\leq X$ which have a representation as a sum of $s$ squares of integers from $\mathcal{E}$.

\begin{corollary}
For $t\geq 2$ an integer and $p>2$ a prime, let $\mathcal{E}$ be a $(p,t,\delta)$-ellipsephic set for some $\delta>0$. Assume that $Y=\#\mathcal{E}(X)\gg X^{1/t}$. Then for $s\geq 3t$ we have
\begin{equation*}
N_{s}(X)\gg X^{1-3\delta/2-\epsilon}.
\end{equation*}
In the case where $\mathcal{E}$ is a $(p,t)^*$-ellipsephic set, we therefore have $N_{s}(X)\gg X^{1-\epsilon}$.
\end{corollary}
\begin{proof}
Using Cauchy's inequality, and writing $R(n)=R_{s,2}^{\mathcal{E}}(n)$ for the number of representations of an integer $n$ as a sum of $s$ squares of integers from $\mathcal{E}$, we have
\begin{align}\label{Rnineq}
\bigg(\sum_{1\leq n\leq X}R(n)\bigg)^2 &\leq \bigg(\sum_{\substack{1\leq n\leq X\\ R(n)>0}}1\bigg)\bigg(\sum_{1\leq n\leq X}R(n)^2\bigg)\nonumber\\
&=N_{s}(X)\sum_{1\leq n\leq X}R(n)^2.
\end{align}
We note that
\begin{align*}
\bigg(\sum_{1\leq n\leq X}R(n)\bigg)^2 &\geq \big(\#\mathcal{E}(\sqrt{X/s})^s\big)^2\\
&\gg Y^{s},
\end{align*}
and, using Theorem \ref{JThm}, that
\begin{align*}
\sum_{1\leq n\leq X}R(n)^2&\ll X^{1/2}J_s(X^{1/2})\\
&\ll X^{1/2}(Y^{1/2})^{2s-3t}(X^{1/2})^{3\delta+\epsilon}\\
&= Y^{s-3t/2}(X^{1/2})^{1+3\delta+\epsilon}.
\end{align*}
Combining these bounds with (\ref{Rnineq}), we see that
\begin{equation*}
N_{s}(X)\gg Y^{3t/2}(X^{-1/2})^{1+3\delta+\epsilon},
\end{equation*}
and our additional assumption on the size of $Y$ allows us to conclude that
\begin{equation*}
N_{s}(X)\gg X^{1-3\delta/2-\epsilon},
\end{equation*}
as required.
\end{proof}

In Section \ref{Prelim} of this paper, we provide a series of preliminary results which form the basis of our iteration process, and in Section \ref{PfMainThm} we complete the proof of Theorem \ref{JThm}.

This paper is based on work appearing in the author's PhD thesis \cite{mythesis} at the University of Bristol, and supported by EPSRC Doctoral Training Partnership EP/M507994/1. During the writing process, she was also supported by the Heilbronn Institute for Mathematical Research, and by the Knut and Alice Wallenberg Foundation (KAW 2018.0362). She would like to thank Trevor Wooley for his supervision and for suggesting this line of research, and Julia Brandes for helpful discussions.

\section{Preliminaries}\label{Prelim}
We recall that we are interested in the integral
\begin{align*}
J(X)=J_{s,2}(X;\bm{\mathfrak{a}})=\oint\bigg|\sum_{x\in\mathcal{E}(X)}\mathfrak{a}_x e(\alpha_1 x+ \alpha_2 x^2)\bigg|^{2s}\,d\bm{\alpha},
\end{align*}
which counts the number of solutions to (\ref{quadVin}) where each solution is counted with weight $\mathfrak{a}_{\bm{x}}\overline{\mathfrak{a}_{\bm{y}}}=\mathfrak{a}_{x_1}\dots\mathfrak{a}_{x_s}\overline{\mathfrak{a}_{y_1}\dots\mathfrak{a}_{y_s}}$. 
We first observe that the case $s>3t$ of Theorem \ref{JThm} follows directly from the case $s=3t$, and so we work only in this latter case throughout.
We also note that it suffices to prove Theorem \ref{JThm} for $X$ a power of $p$ because, for $p^{C-1}<X<p^C$, we then have
\begin{align*}
J(X)\ll J(p^C) \ll (pX)^{3\delta+\epsilon}\bigg(\sum_{x\in\mathcal{E}(p^C)}\left|\mathfrak{a}_x \right|^2\bigg)^s
\end{align*}
for any choice of $\bm{\mathfrak{a}}$, and so
\begin{align*}
J(X) \ll X^{3\delta+\epsilon}\bigg(\sum_{x\in\mathcal{E}(X)}\left|\mathfrak{a}_x \right|^2\bigg)^s,
\end{align*}
since we may assume that $\mathfrak{a}_x=0$ for $x>X$. 

We apply the following normalisation. Let
\begin{equation*}
\rho_0=\bigg(\sum_{x\in\mathcal{E}(X)}\left|\mathfrak{a}_x \right|^2\bigg)^{1/2},
\end{equation*}
and for any $\bm{\alpha}\in [0,1]^2$, let
\begin{equation*}
f(\bm{\alpha})=f(\bm{\alpha};\bm{\mathfrak{a}})=\rho_0^{-1}\sum_{x\in\mathcal{E}(X)}\mathfrak{a}_x e(\alpha_1 x+ \alpha_2 x^2),
\end{equation*}
and define the normalised mean value
\begin{align*}
\mathfrak{J}(X)=\mathfrak{J}_{s,2}(X;\bm{\mathfrak{a}})=\oint\big|f(\bm{\alpha};\bm{\mathfrak{a}})\big|^{2s}\,d\bm{\alpha}=\rho_0^{-2s}J(X).
\end{align*}
Note that this normalisation allows us to assume that $\abs{\mathfrak{a}_x}\leq 1$ for all $x\in\mathcal{E}$. We may also restrict ourselves to the situation in which our weights are real and non-negative, as follows. Let $\mathfrak{a}_x=\mathfrak{b}^+_x-\mathfrak{b}^-_x+i\mathfrak{c}^+_x-i\mathfrak{c}^-_x$, where $\mathfrak{b}^+_x,\mathfrak{b}^-_x,\mathfrak{c}^+_x$ and $\mathfrak{c}^-_x$ are non-negative real numbers, with at most one of $\mathfrak{b}^+_x$ and $\mathfrak{b}^-_x$ non-zero, and at most one of $\mathfrak{c}^+_x$ and $\mathfrak{c}^-_x$ non-zero. Write
\begin{align*}
&g_1(\bm{\alpha})=\sum_{x\in\mathcal{E}(X)}\mathfrak{b}^{+}_x e(\alpha_1 x+ \alpha_2 x^2), &g_2(\bm{\alpha})=\sum_{x\in\mathcal{E}(X)}\mathfrak{b}^{-}_x e(\alpha_1 x+ \alpha_2 x^2),\\
&g_3(\bm{\alpha})=\sum_{x\in\mathcal{E}(X)}\mathfrak{c}^{+}_x e(\alpha_1 x+ \alpha_2 x^2),&g_4(\bm{\alpha})=\sum_{x\in\mathcal{E}(X)}\mathfrak{c}^{-}_x e(\alpha_1 x+ \alpha_2 x^2),
\end{align*}
and observe that
\begin{align*}
\sum_{x\in\mathcal{E}(X)}\mathfrak{a}_x e(\alpha_1 x+ \alpha_2 x^2)=g_1(\bm{\alpha})-g_2(\bm{\alpha})+ig_3(\bm{\alpha})-ig_4(\bm{\alpha})=\sum_{j=1}^4 \epsilon_j g_j(\bm{\alpha}),
\end{align*}
where we have chosen $\epsilon_j\in\{\pm 1,\pm i\}$ appropriately. By H\"older's inequality, we now split up the integrals we are interested in into the parts corresponding to each of these weights, to see that
\begin{align*}
\oint\bigg|\sum_{x\in\mathcal{E}(X)}\mathfrak{a}_x e(\alpha_1 x+ \alpha_2 x^2)\bigg|^{2s}\,d\bm{\alpha}&=\oint\bigg|\sum_{j=1}^4 \epsilon_j g_j(\bm{\alpha})\bigg|^{2s}\,d\bm{\alpha}\ll\max_{j}\oint \big| g_j(\bm{\alpha})\big|^{2s}\,d\bm{\alpha},
\end{align*}
and that since $\left|\mathfrak{b}^{\pm}_x\right|,\left|\mathfrak{c}^{\pm}_x\right|\leq \left|\mathfrak{a}_x\right|$, we obtain the required bounds for general weights from those for real, non-negative weights as claimed. We let
\begin{equation*}
\mathbb{D}=\{\bm{\mathfrak{a}}\mid \mathfrak{a}_x\in [0,1] \mbox{ for all } x\in\mathcal{E} \},
\end{equation*}
and from now on we work with $\bm{\mathfrak{a}}\in\mathbb{D}$.

With the above normalisation, we see that an estimate of the desired form
\begin{equation*}
J(X)\ll X^{\Delta}\bigg(\sum_{x\in\mathcal{E}(X)}\left|\mathfrak{a}_x \right|^2\bigg)^s,
\end{equation*}
for some $\Delta>0$, follows directly from one of the form
\begin{equation*}
\mathfrak{J}(X)\ll X^{\Delta}.
\end{equation*}

We define
\begin{equation*}
\lambda=\sup_{\bm{\mathfrak{a}}\in\mathbb{D}}\limsup_{X\to\infty}\frac{\log{\mathfrak{J}(X;\bm{\mathfrak{a}})}}{\log{X}}.
\end{equation*}
An application of the Cauchy--Schwarz inequality gives us the trivial bound $\lambda\leq s$. Taking into account the expected value of $\lambda$, we define $\Lambda=\lambda -3\delta$ for ease of notation.

We introduce a series of interdependent constants which come into play during the proof of Theorem \ref{JThm} and the results of this section.
Let $\epsilon_0>0$, and suppose $\Lambda>\epsilon_0$. This is the assumption which we ultimately contradict in Section \ref{PfMainThm}.

Let $n=\lceil 16t/\Lambda\rceil$, which will be the number of iterations of the main process in Section \ref{PfMainThm}, and note that the existence and size of $n$ is dependent on our assumption that $\Lambda$ is bounded away from zero. While we would usually expect it to be significantly larger, we certainly have $n\geq 5$. Let $\iota=\lambda/2^{2n+3}$, and observe that by the definition of $\lambda$, there exists a sequence $(X_m)_{m=1}^{\infty}$ tending to infinity with the property that for some $\bm{\mathfrak{a}}\in\mathbb{D}$, and for large enough $m$, we have
\begin{align*}
\mathfrak{J}(X_m;\bm{\mathfrak{a}})>X_m^{\lambda-\iota}.
\end{align*}
Henceforth, we work with a choice of $\bm{\mathfrak{a}}\in\mathbb{D}$ satisfying this condition. In addition, for any $\bm{\mathfrak{b}}\in\mathbb{D}$, we have
\begin{equation*}
\mathfrak{J}(X;\bm{\mathfrak{b}})\ll X^{\lambda+\iota}.
\end{equation*}

Suppose that $X = p^B$, where $B\in\mathbb{N}$ is a large parameter which satisfies $B\geq 2^{n+3}$ and also ensures that $X$ is sufficiently large with regards to the sequence $(X_m)$. The proof of our main theorem features $\nu$ preliminary steps to handle solutions in which variables are congruent modulo small powers of $p$, as well as an initialisation step of size $p^u$, where $\nu$ and $u$ are large in some respects, but small in relation to $B$. Specifically, let $\nu=\lceil B/2^{2n+2} \rceil$ and $u=\lceil B/2^{n+2}\rceil$. We record two further bounds which will come into play in Section \ref{PfMainThm}. We have
\begin{align}\label{bntheta}
2^n (u +1) +\nu -1 &\leq B/4+2^{n+1}+B/{2^{2n+2}} \nonumber\\
&\leq B/2+B/{2^{2n+2}} < B,
\end{align}
and
\begin{align}\label{iotabd}
2tu-\nu&\geq 2tB/2^{n+2}-B/2^{2n+2}-1 \nonumber\\
&\geq (2^{n+1}t-1-2^{n-1})B/2^{2n+2} \nonumber\\
&> \lambda B/2^{2n+2} = 2\iota B.
\end{align}

Our work is heavily dependent on the partition of our variables into congruence classes modulo various powers of the base prime $p$, and we therefore wish to define the restriction of $f(\bm{\alpha})$ to such classes. For $a\in\mathbb{N}$ and $\xi\in\mathcal{E}(p^a)$, let
\begin{equation*}
\rho_a(\xi)=\bigg(\sum_{\substack{x\in\mathcal{E}(X)\\ x\equiv\xi\mmod{p^a}}}\mathfrak{a}_x^2\bigg)^{1/2}
\end{equation*}
and
\begin{equation*}
f_a(\bm{\alpha},\xi)=\rho_a(\xi)^{-1}\sum_{\substack{x\in\mathcal{E}(X)\\ x\equiv\xi\mmod{p^a}}}\mathfrak{a}_x e(\alpha_1 x+ \alpha_2 x^2).
\end{equation*}
For convenience, we let $\rho_0(\xi)=\rho_0$ and $f_0(\bm{\alpha},\xi)=f(\bm{\alpha})$ for any $\xi$. 
We observe that for any $a\in\mathbb{N}$, we have
\begin{equation}\label{sumrhoa}
\sum_{\xi\in\mathcal{E}(p^a)}\rho_a(\xi)^2=\rho_0^2,
\end{equation}
and more generally, for $a,b\in\mathbb{N}$ with $a\leq b$,
\begin{equation*}
\sum_{\substack{\xi'\in\mathcal{E}(p^b)\\\xi'\equiv\xi\mmod{p^a}}}\rho_b(\xi')^2=\rho_a(\xi)^2.
\end{equation*}

Our first lemma provides a useful upper bound required for completion of the proof of Theorem \ref{JThm}. 
\begin{lemma}\label{atheta}
For $a\in\mathbb{N}$ with $p^a<X$, we have
\begin{equation*}
\oint\left|f_a(\bm{\alpha},\xi)\right|^{6t}\,d\bm{\alpha}\ll (X/p^a)^{\lambda+\iota}.
\end{equation*}
\end{lemma}
\begin{proof}
The above integral counts solutions to the system
\begin{equation*}
\sum_{i=1}^{3t}(x_i^2-y_i^2)=0=\sum_{i=1}^{3t}(x_i-y_i)
\end{equation*}
with $x_i,y_i\in\mathcal{E}(X)$ for $1\leq i\leq 3t$ and $\bm{x}\equiv\bm{y}\equiv\xi\mmod{p^a}$, where solutions are counted with weight $\mathfrak{a}_{\bm{x}}\mathfrak{a}_{\bm{y}}\rho_a(\xi)^{-6t}$. Writing $x_i=p^az_i+\xi$ and $y_i=p^aw_i+\xi$ for $1\leq i\leq 3t$, and defining a new set of weights $\mathfrak{b}_z=\rho_a(\xi)^{-1}\mathfrak{a}_{p^az+\xi}$, we can reinterpret the above system in the form
\begin{equation*}
\sum_{i=1}^{3t}(z_i^2-w_i^2)=0=\sum_{i=1}^{3t}(z_i-w_i)
\end{equation*}
with $z_i,w_i\in\mathcal{E}\big((X-\xi)/p^a\big)$ for $1\leq i\leq 3t$ and solutions counted with weight $\mathfrak{b}_{\bm{z}}\mathfrak{b}_{\bm{w}}$. By definition, this is $\mathfrak{J}\big((X-\xi)/p^a;\bm{\mathfrak{b}}\big)$, and consequently we have
\begin{align*}
\oint\left|f_a(\bm{\alpha},\xi)\right|^{6t}\,d\bm{\alpha}&\ll \mathfrak{J}(X/p^a;\bm{\mathfrak{b}})
\ll (X/p^a)^{\lambda+\iota}. \qedhere
\end{align*}
\end{proof}

We want to count solutions to congruences modulo some power ${p^c}$ in the way that we count solutions to equations, via orthogonality, and as such, we make use of Wooley's notation
\begin{equation*}
\oint_{p^c}F(\alpha)\,d\alpha = p^{-c}\sum_{1\leq u\leq p^c}F(u/p^c),
\end{equation*}
and observe that $\oint_{p^c}\left|f(\bm{\alpha})\right|^{2s}\,d\bm{\alpha}$ counts the number of solutions to the system
\begin{equation*}
\sum_{i=1}^s(x_i^j-y_i^j)\equiv 0\mmod{p^c},\quad(j=1,2)
\end{equation*}
with $\bm{x},\bm{y}\in\mathcal{E}(X)^s$, weighted by $\mathfrak{a}_{\bm{x}}\mathfrak{a}_{\bm{y}}\rho_0^{-2s}$.

The next lemma provides the key ``lifting'' step of the process, in which we make use of the $E_t(\delta)$ property of our digit set to raise the power of $p$ used in our congruences. In preparation for this, for $c,d\in\mathbb{N}_0$ with $c\leq d$, weights $\bm{\mathfrak{b}}=(\mathfrak{b}_x)_{x\in\mathcal{E}}$ with $\abs{\mathfrak{b}_x}\leq 1$ for all $x\in\mathcal{E}$, and $\bm{z}\in\mathcal{E}(p^{c})^t$, we define
\begin{equation*}
G_{c,d}(\bm{z})=G_{c,d}(\bm{z},\bm{\mathfrak{b}})=\oint_{p^{d}}\bigg|\sum_{\substack{\bm{x}\in\mathcal{E}(X)^t\\ \bm{x}\equiv\bm{z}\mmod{p^{c}}}}\mathfrak{b}_{\bm{x}}e\big(\beta(x_1+\dots+x_t)\big)\bigg|^2\,d\beta,
\end{equation*}
which counts solutions to the congruence
\begin{equation}\label{dcong}
\sum_{i=1}^t x_i\equiv\sum_{i=1}^t y_i\mmod{p^d}
\end{equation}
with $\bm{x},\bm{y}\in\mathcal{E}(X)^t$ and $\bm{x}\equiv\bm{y}\equiv\bm{z}\mmod{p^c}$, with weight $\mathfrak{b}_{\bm{x}}\overline{\mathfrak{b}_{\bm{y}}}$.

We now show that, up to a small cost, the number of such solutions is essentially controlled by the case in which $\bm{x}\equiv\bm{y}\mmod{p^d}$.

\begin{lemma}\label{diaglemma}
We have
\begin{equation*}
G_{c,d}(\bm{z})\ll p^{\delta(d-c)}\sum_{\substack{\bm{u}\in\mathcal{E}(p^d)^t\\ \bm{u}\equiv\bm{z}\mmod{p^c}}}\Big|\sum_{\substack{\bm{x}\in\mathcal{E}(X)^t\\ \bm{x}\equiv\bm{u}\mmod{p^d}}}\mathfrak{b}_{\bm{x}}\,\Big|^2.
\end{equation*}
\end{lemma}
\begin{proof}
 For $1\leq i\leq t$, let 
 \begin{equation*}
 x_i=z_i +\sum_{r\geq c} x_i^{(r)}p^r
 \end{equation*}
 and
 \begin{equation*}
 y_i=z_i +\sum_{r\geq c} y_i^{(r)}p^r,
 \end{equation*}
with $x_i^{(r)},y_i^{(r)}\in A_p$ for $1\leq i\leq t$ and $r\geq c$. We bound the number of solutions to (\ref{dcong}) by considering each base $p$ digit in turn. Let 
\begin{equation*}
\mathcal{A}_t(h)=\bigg\{\bm{u}\in A_p^{t}\biggm| \sum_{i=1}^t u_i=h\bigg\},
\end{equation*}
and
\begin{equation*}
\widetilde{\mathcal{A}}_t(h)=\bigg\{(\bm{u},\bm{v})\in A_p^{2t}\biggm| \sum_{i=1}^t (u_i-v_i)=h\bigg\}.
\end{equation*}
Summing the lowest digits which interest us (namely, those corresponding to the $p^c$ term in the base $p$ expansion of our variables), we see that a solution of (\ref{dcong}) satisfies
\begin{equation*}
(\bm{x}^{(c)},\bm{y}^{(c)})\in\widetilde{\mathcal{A}}_t(\lambda_c p)
\end{equation*}
for some $1-t\leq \lambda_c\leq t-1$. Accounting for this carry-over, and moving on to the next highest digits, we then see that
\begin{equation*}
(\bm{x}^{(c+1)},\bm{y}^{(c+1)})\in\widetilde{\mathcal{A}}_t(\lambda_{c+1} p-\lambda_c)
\end{equation*}
for some $1-t\leq \lambda_{c+1}\leq t-1$.
Continuing this process, and setting $\lambda_{c-1}=0$ for convenience, we obtain the system
\begin{equation*}
(\bm{x}^{(r)},\bm{y}^{(r)})\in\widetilde{\mathcal{A}}_t(\lambda_r p-\lambda_{r-1}),\quad (c\leq r\leq d-1).
\end{equation*}
For brevity, we use the notation $\underline{\bm{u}}$ to denote the tuple $(\bm{u}^{(c)},\dots,\bm{u}^{(d-1)})$---this represents a regrouping of our variables by digit---and similarly we use $(\underline{\bm{u}},\underline{\bm{v}})$ for $\Big((\bm{u}^{(c)},\bm{v}^{(c)}),\dots,(\bm{u}^{(d-1)},\bm{v}^{(d-1)})\Big)$.

We write
\begin{equation*}
\mathcal{A}_t(\bm{h})=\bigg\{\underline{\bm{u}}\in A_p^{t(d-c)}\biggm| \bm{u}^{(r)}\in{\mathcal{A}}_t(h_r)\mbox{ for }c\leq r\leq d-1\bigg\}
\end{equation*}
and
\begin{equation*}
\widetilde{\mathcal{A}}_t(\bm{h})=\bigg\{(\underline{\bm{u}},\underline{\bm{v}})\in A_p^{2t(d-c)}\biggm| (\bm{u}^{(r)},\bm{v}^{(r)})\in\widetilde{\mathcal{A}}_t(h_r)\mbox{ for }c\leq r\leq d-1\bigg\},
\end{equation*}
and observe that these are the sets of all possible variables with given digit sums. By convention, we suppose that for any $\underline{\bm{u}}=(\bm{u}^{(c)},\dots,\bm{u}^{(d-1)})\in\mathcal{A}_t(\bm{h})$, we have $u_i=z_i+\sum_{c\leq r\leq d-1}u_i^{(r)}p^r$ and write $\bm{u}=(u_1,\dots,u_t)$, and similarly for $(\underline{\bm{u}},\underline{\bm{v}})\in\widetilde{\mathcal{A}}_t(\bm{h})$.

For $\bm{\lambda}=(\lambda_c,\dots,\lambda_{d-1})\in\{1-t,\dots,t-1\}^{d-c}$, we write 
\begin{equation*}
\bm{\lambda'}=(\lambda_c p-\lambda_{c-1},\dots,\lambda_{d-1} p-\lambda_{d-2}). 
\end{equation*}
We are now in a position to observe that
\begin{align*}
G_{c,d}(\bm{z})= \sum_{\bm{\lambda}\in\{1-t,\dots,t-1\}^{d-c}} \sum_{(\underline{\bm{u}},\underline{\bm{v}})\in\widetilde{\mathcal{A}}_t(\bm{\lambda'})}\sum_{\substack{\bm{x},\bm{y}\in\mathcal{E}(X)^t\\ (\bm{x},\bm{y})\equiv(\bm{u},\bm{v})\mmod{p^d}}}\mathfrak{b}_{\bm{x}}\overline{\mathfrak{b}_{\bm{y}}}.
\end{align*}
Writing
\begin{equation*}
\mathfrak{B}(\bm{u})=\sum_{\substack{\bm{x}\in\mathcal{E}(X)^t\\ \bm{x}\equiv\bm{u}\mmod{p^d}}}\mathfrak{b}_{\bm{x}}
\end{equation*}
and
\begin{equation*}
\phi_{\bm{u}}(\bm{\gamma})=\gamma_c\sum_{i=1}^t u_i^{(c)}+\dots+\gamma_{d-1}\sum_{i=1}^t u_i^{(d-1)}
\end{equation*}
for brevity, and encoding the condition $(\underline{\bm{u}},\underline{\bm{v}})\in\widetilde{\mathcal{A}}_t(\bm{\lambda'})$ in integral form, we see that
\begin{align*}
\sum_{(\underline{\bm{u}},\underline{\bm{v}})\in\widetilde{\mathcal{A}}_t(\bm{\lambda'})}\hspace{-0.1in}\mathfrak{B}(\bm{u})\overline{\mathfrak{B}(\bm{v})}&=\sum_{\substack{(\underline{\bm{u}},\underline{\bm{v}}) \in A_p^{2t(d-c)}}}\hspace{-0.1in}\mathfrak{B}(\bm{u})\overline{\mathfrak{B}(\bm{v})} \oint e\big(\phi_{\bm{u}}(\bm{\gamma})-\phi_{\bm{v}}(\bm{\gamma})\big) e(-\bm{\gamma}\cdot\bm{\lambda'})\,d\bm{\gamma}\\
&=\oint e(-\bm{\gamma}\cdot\bm{\lambda'})\sum_{\substack{(\underline{\bm{u}},\underline{\bm{v}}) \in A_p^{2t(d-c)}}}\hspace{-0.1in} \mathfrak{B}(\bm{u})\overline{\mathfrak{B}(\bm{v})} e\big(\phi_{\bm{u}}(\bm{\gamma})-\phi_{\bm{v}}(\bm{\gamma})\big) \,d\bm{\gamma}\\
&\leq \oint\bigg|\sum_{\substack{(\underline{\bm{u}},\underline{\bm{v}}) \in A_p^{2t(d-c)}}}\hspace{-0.1in} \mathfrak{B}(\bm{u})\overline{\mathfrak{B}(\bm{v})} e\big(\phi_{\bm{u}}(\bm{\gamma})-\phi_{\bm{v}}(\bm{\gamma})\big)\bigg| \,d\bm{\gamma}.
\end{align*}
The expression on the right-hand side is now independent of our choice of $\bm{\lambda}$, so we conclude that
\begin{align*}
G_{c,d}(\bm{z})&\leq (2t-1)^{d-c} \oint \bigg|\sum_{\substack{\underline{\bm{u}} \in A_p^{t(d-c)}}}\hspace{-0.1in} \mathfrak{B}(\bm{u}) e(\phi_{\bm{u}}(\bm{\gamma}))\bigg|^2 \,d\bm{\gamma}\\
&\ll \sum_{(\underline{\bm{u}},\underline{\bm{v}})\in\widetilde{\mathcal{A}}_t(\bm{0})}\sum_{\substack{\bm{x},\bm{y}\in\mathcal{E}(X)^t\\ (\bm{x},\bm{y})\equiv(\bm{u},\bm{v})\mmod{p^d}}}\mathfrak{b}_{\bm{x}}\overline{\mathfrak{b}_{\bm{y}}}\\
&= \sum_{0\leq\bm{n}\leq t(p-1)}\bigg|\sum_{\underline{\bm{u}}\in\mathcal{A}_t(\bm{n})}\sum_{\substack{\bm{x}\in\mathcal{E}(X)^t\\ \bm{x}\equiv\bm{u}\mmod{p^d}}}\mathfrak{b}_{\bm{x}}\,\bigg|^2.
\end{align*}
Using Cauchy's inequality, we see that
\begin{align*}
G_{c,d}(\bm{z})&\ll \sum_{0\leq\bm{n}\leq t(p-1)}\bigg(\sum_{\underline{\bm{u}}\in\mathcal{A}_t(\bm{n})}\Big|\sum_{\substack{\bm{x}\in\mathcal{E}(X)^t\\ \bm{x}\equiv\bm{u}\mmod{p^d}}}\mathfrak{b}_{\bm{x}}\,\Big|^2\bigg)\bigg(\sum_{\underline{\bm{u}}\in\mathcal{A}_t(\bm{n})}1\bigg).
\end{align*}
From our initial assumption that $\mathcal{E}$ is a $(p,t,\delta)$-ellipsephic set, we know that for $\bm{n}=(n_c,\dots,n_{d-1})$ with $0\leq\bm{n}\leq t(p-1)$, we have
\begin{equation*}
\#\mathcal{A}_t(\bm{n})=\#\bigg\{\underline{\bm{u}}\in A_p^{t(d-c)}\biggm| \sum_{i=1}^t u_i^{(r)}=n_r\mbox{ for }c\leq r\leq d-1\bigg\}
\ll\prod_{r=c}^{d-1} p^{\delta}= p^{\delta(d-c)},
\end{equation*}
and consequently
\begin{align*}
G_{c,d}(\bm{z})&\ll p^{\delta(d-c)} \sum_{0\leq\bm{n}\leq t(p-1)}\sum_{\underline{\bm{u}}\in\mathcal{A}_t(\bm{n})}\Big|\sum_{\substack{\bm{x}\in\mathcal{E}(X)^t\\ \bm{x}\equiv\bm{u}\mmod{p^d}}}\mathfrak{b}_{\bm{x}}\,\Big|^2\\
&=p^{\delta(d-c)}\sum_{\substack{\bm{u}\in\mathcal{E}(p^d)^t\\ \bm{u}\equiv\bm{z}\mmod{p^c}}}\Big|\sum_{\substack{\bm{x}\in\mathcal{E}(X)^t\\ \bm{x}\equiv\bm{u}\mmod{p^d}}}\mathfrak{b}_{\bm{x}}\,\Big|^2,
\end{align*}
as claimed. \qedhere
\end{proof}

The first of the expressions we are interested in represents the weighted number of solutions to our system of equations in which the variables fall into certain congruence classes modulo powers of $p$. For $a,b\in\mathbb{N}$, we let
\begin{equation*}
I_{a,b}(\xi,\eta)=\oint\left|f_a(\bm{\alpha},\xi)\right|^{2t}\left|f_b(\bm{\alpha},\eta)\right|^{4t}\,d\bm{\alpha},
\end{equation*}
and observe that this expression counts the number of solutions to the system
\begin{equation}\label{Iabcong}
\sum_{i=1}^t(x_i^j-y_i^j)=\sum_{l=1}^{2t}(u_l^j-v_l^j),\quad(j=1,2)
\end{equation}
with $x_i,y_i,u_l,v_l\in\mathcal{E}(X)$ for $1\leq i\leq t$ and $1\leq l\leq 2t$, satisfying $\bm{x}\equiv\bm{y}\equiv\xi\mmod{p^a}$ and $\bm{u}\equiv\bm{v}\equiv\eta\mmod{p^b}$, and with each solution being counted with weight $\rho_a(\xi)^{-2t}\rho_b(\eta)^{-4t} \mathfrak{a}_{\bm{x}}\mathfrak{a}_{\bm{y}}\mathfrak{a}_{\bm{u}}\mathfrak{a}_{\bm{v}}$.
We also assume that $I_{0,0}(\xi,\eta)=\mathfrak{J}(X)$ for any $\xi$ and $\eta$.

Next, a weighted sum over the possible values of $\xi$ and $\eta$ in the above definition will simplify later computations. For $h\in\mathbb{N}$, we define
\begin{equation}\label{Kdefn}
K_{a,b}^{h}=\rho_0^{-4}\sum_{\xi\in\mathcal{E}(p^a)}\sum_{\substack{\eta\in\mathcal{E}(p^b)\\ p^{h-1}\|(\xi-\eta)}}\rho_a(\xi)^2\rho_b(\eta)^2 I_{a,b}(\xi,\eta),
\end{equation}
where the notation $p^c\| d$ means that $p^c\mid d$ and $p^{c+1}\nmid d$.

The next lemma allows us to apply Lemma \ref{diaglemma} as the key ingredient in an iterative process which we use in Section \ref{PfMainThm} to complete the proof of Theorem \ref{JThm}. 

\begin{lemma}\label{Kiter}
For $a,b,h\in\mathbb{N}$ satisfying $h\leq a<b\leq 2a-h+1$ and $p^b<X$, 
we have
\begin{equation*}
K_{a,b}^{h}\ll p^{\delta(2b-a-h+1)} (X/p^b)^{(\lambda+\iota)/2}(K_{b,2b-h+1}^{h})^{1/2}.
\end{equation*}
\end{lemma}
\begin{proof}
We begin by considering $I_{a,b}(\xi,\eta)$, and note that by the definition of $K_{a,b}^{h}$, we may assume that we are working in the situation in which $p^{h-1}\|(\xi-\eta)$.

Writing $x_i=p^a\tilde{x}_i+\xi$ and $u_l=p^b\tilde{u}_l+\eta$, and similarly for $\bm{y}$ and $\bm{v}$, we apply the binomial theorem to (\ref{Iabcong}) to see that
\begin{equation*}
\sum_{i=1}^t\big((p^a\tilde{x}_i+\xi-\eta)^j-(p^a\tilde{y}_i+\xi-\eta)^j\big)=p^{jb}\sum_{l=1}^{2t}(\tilde{u}_l^j-\tilde{v}_l^j),\quad(j=1,2),
\end{equation*}
and consequently that we have the congruences
\begin{equation*}
\sum_{i=1}^t\big((p^a\tilde{x}_i+\xi-\eta)^j-(p^a\tilde{y}_i+\xi-\eta)^j\big)\equiv 0\mmod{p^{jb}},\quad(j=1,2).
\end{equation*}
In other words, we have
\begin{equation}\label{cong1}
\sum_{i=1}^t(\tilde{x}_i-\tilde{y}_i)\equiv 0\mmod{p^{b-a}},
\end{equation}
and
\begin{equation}\label{cong2}
p^a\sum_{i=1}^t(\tilde{x}_i^2-\tilde{y}_i^2)+2(\xi-\eta)\sum_{i=1}^t(\tilde{x}_i-\tilde{y}_i)\equiv 0\mmod{p^{2b-a}}.
\end{equation}

We fix the weights appearing in the definition of $G_{c,d}(\bm{z})$ to be
\begin{equation*}
\mathfrak{b}_x=\rho_a(\xi)^{-1}\mathfrak{a}_xe(\alpha_1x+\alpha_2x^2).
\end{equation*}
Encoding (\ref{cong1}) as part of our integral, and writing $\bm{\xi}=(\xi,\dots,\xi)$, we have
\begin{align*}
I_{a,b}(\xi,\eta)&=\oint G_{a,b}(\bm{\xi})\left|f_b(\bm{\alpha},\eta)\right|^{4t}\,d\bm{\alpha}.
\end{align*}
By Lemma \ref{diaglemma}, we may conclude that
\begin{equation*}
I_{a,b}(\xi,\eta)\ll p^{\delta(b-a)}\sum_{\substack{\bm{z}\in\mathcal{E}(p^{b})^t\\\bm{z}\equiv\xi\mmod{p^a}}}\oint
\Big|\sum_{\substack{\bm{x}\in\mathcal{E}(X)^t\\\bm{x}\equiv\bm{z}\mmod{p^b}}}\mathfrak{b}_{\bm{x}}\hspace{0.05in}\Big|^2\left|f_b(\bm{\alpha},\eta)\right|^{4t}\,d\bm{\alpha}.
\end{equation*}
We have therefore introduced, at a cost of $p^{\delta(b-a)}$, the additional condition
\begin{equation*}
{x}_i\equiv{y}_i\mmod{p^{b}},\quad(1\leq i\leq t),
\end{equation*}
or equivalently
\begin{equation*}
\tilde{x}_i\equiv\tilde{y}_i\mmod{p^{b-a}},\quad(1\leq i\leq t).
\end{equation*}
Substituting this back into (\ref{cong2}), and using the facts that $p^{h-1}\|(\xi -\eta)$ and $h-1<a<b$, we see that
\begin{equation*}
\sum_{i=1}^t(\tilde{x}_i-\tilde{y}_i)\equiv 0\mmod{p^{b-h+1}}.
\end{equation*}
Encoding this congruence as before, we obtain
\begin{equation*}
I_{a,b}(\xi,\eta)\ll p^{\delta(b-a)}\sum_{\substack{\bm{z}\in\mathcal{E}(p^{b})^t\\\bm{z}\equiv\xi\mmod{p^a}}}\oint
G_{b,a+b-h+1}(\bm{z})\left|f_b(\bm{\alpha},\eta)\right|^{4t}\,d\bm{\alpha}.
\end{equation*}
We now apply Lemma \ref{diaglemma} again to see that
\begin{equation*}
I_{a,b}(\xi,\eta)\ll p^{\delta(b-h+1)}\sum_{\substack{\bm{z}\in\mathcal{E}(p^{a+b-h+1})^t\\\bm{z}\equiv\xi\mmod{p^a}}}\oint
\Big|\hspace{-0.25in}\sum_{\substack{\bm{x}\in\mathcal{E}(X)^t\\\bm{x}\equiv\bm{z}\mmod{p^{a+b-h+1}}}}\hspace{-0.25in}\mathfrak{b}_{\bm{x}}\hspace{0.05in}\Big|^2\left|f_b(\bm{\alpha},\eta)\right|^{4t}\,d\bm{\alpha},
\end{equation*}
and we have introduced the additional condition
\begin{equation*}
\tilde{x}_i\equiv\tilde{y}_i\mmod{p^{b-h+1}},\quad(1\leq i\leq t).
\end{equation*}
Repeating this process, we reach the situation in which
\begin{equation*}
\sum_{i=1}^t(\tilde{x}_i-\tilde{y}_i)\equiv 0\mmod{p^{2b-a-h+1}},
\end{equation*}
and a final application of Lemma \ref{diaglemma} allows us to conclude that
\begin{equation*}
I_{a,b}(\xi,\eta)\ll p^{\delta(2b-a-h+1)}\sum_{\substack{\bm{z}\in\mathcal{E}(p^{2b-h+1})^t\\\bm{z}\equiv\xi\mmod{p^a}}}\oint
\Big|\hspace{-0.2in}\sum_{\substack{\bm{x}\in\mathcal{E}(X)^t\\\bm{x}\equiv\bm{z}\mmod{p^{2b-h+1}}}}\hspace{-0.2in}\mathfrak{b}_{\bm{x}}\hspace{0.05in}\Big|^2\left|f_b(\bm{\alpha},\eta)\right|^{4t}\,d\bm{\alpha}.
\end{equation*}

Using the definition of the weights $\bm{\mathfrak{b}}$, and writing $b'=2b-h+1$, we deduce that
\begin{align*}
I_{a,b}(\xi,\eta)&\ll p^{\delta(b'-a)}\oint\Big(\sum_{\substack{\xi'\in\mathcal{E}(p^{b'})\\\xi'\equiv\xi\mmod{p^a}}}\rho_a(\xi)^{-2}\rho_{b'}(\xi')^{2}\left|f_{b'}(\bm{\alpha},\xi')\right|^{2}\Big)^t\left|f_b(\bm{\alpha},\eta)\right|^{4t}\,d\bm{\alpha},
\end{align*}
and note that our assumption that $p^{h-1}\|(\xi-\eta)$ implies that we also have $p^{h-1}\|(\xi'-\eta)$. An application of H\"older's inequality gives
\begin{align*}
I_{a,b}(\xi,\eta)&\ll p^{\delta(b'-a)}\rho_a(\xi)^{-2}\sum_{\substack{\xi'\in\mathcal{E}(p^{b'})\\\xi'\equiv\xi\mmod{p^a}}}\rho_{b'}(\xi')^{2}\oint\left|f_{b'}(\bm{\alpha},\xi')\right|^{2t}\left|f_b(\bm{\alpha},\eta)\right|^{4t}\,d\bm{\alpha}\\
&= p^{\delta(b'-a)}\rho_a(\xi)^{-2}\sum_{\substack{\xi'\in\mathcal{E}(p^{b'})\\\xi'\equiv\xi\mmod{p^a}}}\rho_{b'}(\xi')^{2}I_{b',b}(\xi',\eta).
\end{align*}
Using Cauchy's inequality and Lemma \ref{atheta}, we see that
\begin{align*}
I_{b',b}(\xi',\eta)&=\oint\left|f_{b'}(\bm{\alpha},\xi')\right|^{2t}\left|f_b(\bm{\alpha},\eta)\right|^{4t}\,d\bm{\alpha}\\
&\leq \bigg(\oint\left|f_b(\bm{\alpha},\eta)\right|^{2t}\left|f_{b'}(\bm{\alpha},\xi')\right|^{4t}\,d\bm{\alpha}\bigg)^{1/2}\bigg(\oint\left|f_b(\bm{\alpha},\eta)\right|^{6t}\,d\bm{\alpha}\bigg)^{1/2}\\
&\ll I_{b,b'}(\eta,\xi')^{1/2} (X/p^b)^{(\lambda+\iota)/2}.
\end{align*}
Substituting this into (\ref{Kdefn}), we see that
\begin{align*}
K_{a,b}^{h}&=\rho_0^{-4}\sum_{\xi\in\mathcal{E}(p^a)}\sum_{\substack{\eta\in\mathcal{E}(p^b)\\ p^{h-1}\|(\xi-\eta)}}\rho_a(\xi)^2\rho_b(\eta)^2 I_{a,b}(\xi,\eta)\\
&\ll p^{\delta(b'-a)} (X/p^b)^{(\lambda+\iota)/2}\rho_0^{-4}\sum_{\eta\in\mathcal{E}(p^b)}\sum_{\substack{\xi'\in\mathcal{E}(p^{b'})\\ p^{h-1}\|(\xi'-\eta)}}\rho_b(\eta)^2 \rho_{b'}(\xi')^{2}I_{b,b'}(\eta,\xi')^{1/2}.
\end{align*}
By Cauchy's inequality and (\ref{sumrhoa}), we conclude that
\begin{align*}
K_{a,b}^h&\ll p^{\delta(b'-a)} (X/p^b)^{(\lambda+\iota)/2}\rho_0^{-2} \bigg (\sum_{\eta\in\mathcal{E}(p^b)}\sum_{\substack{\xi'\in\mathcal{E}(p^{b'})\\ p^{h-1}\|(\xi'-\eta)}}\rho_b(\eta)^2 \rho_{b'}(\xi')^{2}I_{b,b'}(\eta,\xi')\bigg)^{1/2}\\
&=p^{\delta(b'-a)} (X/p^b)^{(\lambda+\iota)/2}(K_{b,b'}^{h})^{1/2},
\end{align*}
as claimed.
\end{proof}

Finally, the following lemma provides a key step in the iterative process of Section \ref{PfMainThm}.
\begin{lemma}\label{Ivvlemma}
For $h\in\mathbb{N}$, and for $\xi\in\mathcal{E}(p^{h-1})$, we have
\begin{align*}
I_{h-1,h-1}&(\xi,\xi)\ll \rho_{h-1}(\xi)^{-4}\bigg(\hspace{-0.1in}\sum_{\substack{\eta\in\mathcal{E}(p^h)\\ \eta\equiv\xi\mmod{p^{h-1}}}}\rho_h(\eta)^{4}I_{h,h}(\eta,\eta)+ p^{2s-2}\rho_0^4 \,K_{h,h}^h\bigg).
\end{align*}
\end{lemma}
\begin{proof}
We observe that
\begin{align*}
I_{h-1,h-1}(\xi,\xi)&=\oint\left|f_{h-1}(\bm{\alpha},\xi)\right|^{2t}\left|f_{h-1}(\bm{\alpha},\xi)\right|^{4t}\,d\bm{\alpha}\\
&=\oint\left|f_{h-1}(\bm{\alpha},\xi)\right|^{2s}\,d\bm{\alpha},
\end{align*}
which counts the number of solutions to (\ref{quadVin}) with $x_i,y_i\in\mathcal{E}(X)$ for $1\leq i\leq s$ and $\bm{x}\equiv\bm{y}\equiv\xi\mmod{p^{h-1}}$, each solution being counted with weight $\rho_{h-1}(\xi)^{-2s}\mathfrak{a}_{\bm{x}}\mathfrak{a}_{\bm{y}}$.

We partition the solutions based on the congruence classes in which the variables lie modulo $p^h$, letting $\mathfrak{J}_h(X,\xi)$ denote the contribution from solutions in which all variables are congruent modulo $p^h$, and $\mathfrak{J}_h^*(X,\xi)$ the contribution from the remaining solutions, so that
\begin{equation}\label{Jsplit}
I_{h-1,h-1}(\xi,\xi)=\mathfrak{J}_h(X,\xi)+ \mathfrak{J}_h^*(X,\xi).
\end{equation} 
We have
\begin{align}\label{Jhest}
\mathfrak{J}_h(X,\xi)&=\sum_{\substack{\eta\in\mathcal{E}(p^h)\\ \eta\equiv\xi\mmod{p^{h-1}}}}\rho_{h-1}(\xi)^{-2s}\rho_h(\eta)^{2s}I_{h,h}(\eta,\eta)\nonumber\\
&\leq \rho_{h-1}(\xi)^{-4}\sum_{\substack{\eta\in\mathcal{E}(p^h)\\ \eta\equiv\xi\mmod{p^{h-1}}}}\rho_h(\eta)^{4}I_{h,h}(\eta,\eta),
\end{align}
since $\rho_h(\eta)^2\leq\rho_{h-1}(\xi)^2$ for $\eta\equiv\xi\mmod{p^{h-1}}$.

When estimating $\mathfrak{J}_h^*(X,\xi)$, we may assume, up to a combinatorial factor, that $x_1\not\equiv x_2\mmod{p^h}$, and observe that $\mathfrak{J}_h^*(X,\xi)$ is bounded above by at most a constant multiple of
\begin{align*}
&\rho_{h-1}(\xi)^{-2}\hspace{-0.2in}\sum_{\substack{\eta\neq\eta'\in\mathcal{E}(p^h)\\ \eta\equiv\eta'\equiv\xi\mmod{p^{h-1}}}}\hspace{-0.2in}\rho_h(\eta)\rho_h(\eta')\oint f_h(\bm{\alpha},\eta)f_h(\bm{-\alpha},\eta')\left|f_{h-1}(\bm{\alpha},\xi)\right|^{2s-2}\,d\bm{\alpha}\\
&\leq \rho_{h-1}(\xi)^{-2}\hspace{-0.2in}\sum_{\substack{\eta\neq\eta'\in\mathcal{E}(p^h)\\ \eta\equiv\eta'\equiv\xi\mmod{p^{h-1}}}}\hspace{-0.2in}\rho_h(\eta)\rho_h(\eta') I_{h,h}(\eta,\eta')^{1/2s} I_{h,h}(\eta',\eta)^{1/2s}I_{h-1,h-1}(\xi,\xi)^{1-1/s},
\end{align*}
by H\"older's inequality. If $\mathfrak{J}_h^*(X,\xi)=\max{\{\mathfrak{J}_h(X,\xi),\mathfrak{J}_h^*(X,\xi)\}}$, we have
\begin{equation*}
I_{h-1,h-1}(\xi,\xi)\ll \mathfrak{J}_h^*(X,\xi),
\end{equation*}
and may rearrange to obtain
\begin{align}\label{Jh*est}
I&_{h-1,h-1}(\xi,\xi)\ll \rho_{h-1}(\xi)^{-2s}\bigg(\sum_{\substack{\eta\neq\eta'\in\mathcal{E}(p^h)\\ \eta\equiv\eta'\equiv\xi\mmod{p^{h-1}}}}\hspace{-0.1in}\rho_h(\eta)\rho_h(\eta') I_{h,h}(\eta,\eta')^{1/s}\bigg)^s\nonumber\\
&\ll\rho_{h-1}(\xi)^{-2s}\bigg(\sum_{\substack{\eta\neq\eta'\in\mathcal{E}(p^h)\\ \eta\equiv\eta'\equiv\xi\mmod{p^{h-1}}}}\hspace{-0.1in}\rho_h(\eta)^s\rho_h(\eta')^s I_{h,h}(\eta,\eta')\bigg) \bigg(\sum_{\substack{\eta\neq\eta'\in\mathcal{E}(p^h)\\ \eta\equiv\eta'\equiv\xi\mmod{p^{h-1}}}}\hspace{-0.1in}1\bigg)^{s-1}\nonumber\\
&\ll \rho_{h-1}(\xi)^{-4}p^{2s-2}\rho_0^4 \,K_{h,h}^h.
\end{align}
Substituting (\ref{Jhest}) and (\ref{Jh*est}) into (\ref{Jsplit}), we deduce that
\begin{align*}
I_{h-1,h-1}&(\xi,\xi)\ll \rho_{h-1}(\xi)^{-4}\bigg(\hspace{-0.1in}\sum_{\substack{\eta\in\mathcal{E}(p^h)\\ \eta\equiv\xi\mmod{p^{h-1}}}}\rho_h(\eta)^{4}I_{h,h}(\eta,\eta)+ p^{2s-2}\rho_0^4 \,K_{h,h}^h\bigg),
\end{align*}
as claimed.
\end{proof}

\section{Proof of Theorem \ref{JThm}} \label{PfMainThm}

We first wish to handle those solutions in which all of our variables are congruent modulo some small power of $p$, since these should contribute neglibly to the total, but would prevent some of the mechanisms of the previous section from working smoothly.

Applying Lemma \ref{Ivvlemma} twice, we have
\begin{align*}
&\mathfrak{J}(X)\ll \rho_0^{-4}\sum_{\xi\in\mathcal{E}(p)}\rho_1(\xi)^{4}I_{1,1}(\xi,\xi) + p^{2s-2} K_{1,1}^1\\
&\ll \rho_0^{-4}\sum_{\xi\in\mathcal{E}(p)}\bigg(\hspace{-0.1in}\sum_{\substack{\eta\in\mathcal{E}(p^{2})\\ \eta\equiv\xi\mmod{p}}}\hspace{-0.1in}\rho_{2}(\eta)^4 I_{2,2}(\eta,\eta)+p^{2s-2}\rho_0^4\,K_{2,2}^{2}\bigg) + p^{2s-2} K_{1,1}^1\\
&= \rho_0^{-4}\sum_{\eta\in\mathcal{E}(p^2)}\rho_2(\eta)^{4}I_{2,2}(\eta,\eta) + p^{2s-1} K_{2,2}^2+p^{2s-2} K_{1,1}^1.
\end{align*}
Repeated application of Lemma \ref{Ivvlemma} therefore yields
\begin{align*}
\mathfrak{J}(X)\ll \rho_0^{-4}\sum_{\omega\in\mathcal{E}(p^{\nu})}\rho_{\nu}(\omega)^{4}I_{\nu,\nu}(\omega,\omega) + \sum_{1\leq h\leq \nu}p^{2s-3+h} K_{h,h}^{h}.
\end{align*}
We have
\begin{align*}
I_{\nu,\nu}(\omega,\omega)&=\oint\left|f_{\nu}(\bm{\alpha},\omega)\right|^{2t}\left|f_{\nu}(\bm{\alpha},\omega)\right|^{4t}\,d\bm{\alpha}\\
&=\oint\left|f_{\nu}(\bm{\alpha},\omega)\right|^{6t}\,d\bm{\alpha} \ll (X/p^{\nu})^{\lambda+\iota},
\end{align*}
by Lemma \ref{atheta}. Consequently, by the definitions of $\nu$ and $\iota$, we have
\begin{align*}
\rho_0^{-4}\sum_{\omega\in\mathcal{E}(p^{\nu})}\rho_{\nu}(\omega)^{4}I_{\nu,\nu}(\omega,\omega) &\ll (X/p^{\nu})^{\lambda+\iota}\rho_0^{-4}\sum_{\omega\in\mathcal{E}(p^{\nu})}\rho_{\nu}(\omega)^{4}\\
&\ll X^{\lambda+\iota}p^{-(\lambda+\iota)B/2^{2n+2}}\\
&= X^{\lambda+\iota-(\lambda+\iota)/2^{2n+2}}=o(X^{\lambda-\iota}).
\end{align*}
By our choice of $\bm{\mathfrak{a}}\in\mathbb{D}$, and the discussions at the beginning of Section \ref{Prelim}, there is consequently some value of $h$ with $1\leq h\leq \nu$ such that
\begin{align*}
\mathfrak{J}(X)\ll  \nu p^{2s-3+h} K_{h,h}^{h}.
\end{align*}
By H\"older's inequality, we have
\begin{align*}
K_{h,h}^{h} \leq p^{(u-1)(2t-1)+u(4t-1)}K_{h+u-1,h+u}^h,
\end{align*}
and consequently
\begin{align}\label{Jh2hbd}
\mathfrak{J}(X)\ll \nu p^{4t+6tu-2u+h} K_{h+u-1,h+u}^{h}.
\end{align}

We define a sequence of indices by the following recurrence relations:
\begin{equation*}
a_0=h+u-1,\quad b_0=h+u,\quad a_m=b_{m-1},\quad b_m=2b_{m-1}-h+1.
\end{equation*}
For convenience we note that $b_m=2^m(u+1)+h-1$.
By Lemma \ref{Kiter}, while $p^{b_m}<X$, which is ensured by (\ref{bntheta}) for $m\leq n$, we have
\begin{align*}
K_{a_m,b_m}^{h} &\ll p^{\delta(2b_m-a_m-h+1)} (X/p^{b_m})^{(\lambda+\iota)/2}(K_{a_{m+1},b_{m+1}}^{h})^{1/2},
\end{align*}
which gives
\begin{align*}
K_{a_0,b_0}^{h} &\ll p^{\delta(u+2)} (X/p^{b_0})^{(\lambda+\iota)/2}(K_{a_{1},b_{1}}^{h})^{1/2},
\end{align*}
and, for $m\geq 1$,
\begin{align*}
K_{a_m,b_m}^{h} & \ll p^{3\cdot 2^{m-1}(u+1)\delta} (X/p^{b_m})^{(\lambda+\iota)/2}(K_{a_{m+1},b_{m+1}}^{h})^{1/2}.
\end{align*}
By iterating this relation, we see that
\begin{align*}
K_{h+u-1,h+u}^{h}&\ll p^{\delta(u+2+3(u+1)(n-1)/2)-n(\lambda+\iota) (u+1)/2}X^{(\lambda+\iota)(1-1/2^n)}(K_{a_{n},b_{n}}^{h})^{1/2^n}\\
&\ll p^{-\delta(u-1)/2+3\delta n(u+1)/2-\lambda n(u+1)/2}X^{(\lambda+\iota)(1-1/2^n)}(K_{a_{n},b_{n}}^{h})^{1/2^n},
\end{align*}
and using the definitions of $\Lambda$ and $n$, we deduce that
\begin{align*}
K_{h+u-1,h+u}^{h} 
&\ll p^{-n\Lambda(u+1)/2}X^{(\lambda+\iota)(1-1/2^n)}(K_{a_{n},b_{n}}^{h})^{1/2^n}\\
&\ll p^{-8t(u+1)}X^{(\lambda+\iota)(1-1/2^n)}(K_{a_{n},b_{n}}^{h})^{1/2^n}.
\end{align*}
Substituting this into (\ref{Jh2hbd}), we see that
\begin{align}\label{nextJbd}
\mathfrak{J}(X)&\ll \nu p^{-4t-2tu-2u+h} X^{(\lambda+\iota)(1-1/2^n)}(K_{a_{n},b_{n}}^{h})^{1/2^n}\nonumber\\
&\ll p^{\nu-2tu-2u}X^{(\lambda+\iota)(1-1/2^n)}(K_{a_{n},b_{n}}^{h})^{1/2^n}\log{X}.
\end{align}
A trivial bound gives us
$K_{a_{n},b_{n}}^{h}\ll X^{\lambda+\iota}.$
Combining this with (\ref{nextJbd}), and using (\ref{iotabd}), we obtain
\begin{align*}
\mathfrak{J}(X)&\ll p^{\nu-2tu-2u}X^{\lambda+\iota+\epsilon} \\
&\ll p^{-2\iota B-2u} X^{\lambda+\iota+\epsilon}\\
&\ll X^{\lambda-\iota-1/{2^{n+1}}+\epsilon}
= o(X^{\lambda-\iota}),
\end{align*}
which provides the required contradiction and completes the proof of Theorem \ref{JThm}. \qed

\newcommand{\noop}[1]{}

\end{document}